\RequirePackage{fix-cm}
\documentclass[letterpaper,notitlepage,11pt]{amsart}
\usepackage{microtype}
\usepackage{fixltx2e}[2006/09/13]
\usepackage{amssymb}
\usepackage{mathtools}
\usepackage[firstinits]{biblatex}

\usepackage{setspace}
\usepackage{xspace}

\usepackage{enumitem}

\usepackage{tikz}
\usetikzlibrary{turtle}
\newtheorem{thm}{Theorem}

\newtheorem{clm}{Claim}
\newtheorem{crl}{Corollary}

\newcommand{\A}{\mathbb{A}}
\newcommand{\N}{\mathbb{N}}
\newcommand{\Z}{\mathbb{Z}}
\newcommand{\Q}{\mathbb{Q}}
\newcommand{\R}{\mathbb{R}}
\newcommand{\C}{\mathbb{C}}
\newcommand{\E}{\mathbb{E}}
\newcommand{\B}{\mathbb{B}}
\newcommand{\F}{\mathbb{F}}
\renewcommand{\P}{\mathbb{P}}

\newcommand{\sF}{\mathcal{F}}

\newcommand{\sL}{\mathcal{L}}

\newcommand{\sT}{\mathcal{T}}

\newcommand{\1}{\mathbf{1}}

\DeclarePairedDelimiter{\abs}{\lvert}{\rvert}  % absolute value, cardinality

\DeclarePairedDelimiterX{\scprd}[2]{\langle}{\rangle}{#1,#2}
\DeclarePairedDelimiterX{\set}[2]{\{}{\}}{\,#1 \mathrel{}\mathclose{}\delimsize|\mathopen{}\mathrel{} #2\,}

\newcommand*{\vsch}[2]{\genfrac{\{}{\}}{0pt}{}{#1}{#2}}
\newcommand*{\stak}[2]{\genfrac{}{}{0pt}{}{#1}{#2}}

\newcommand{\res}{\kern -0.1em\mathpunct{\upharpoonright}}

\renewcommand{\multicitedelim}{; }

\title{Two-orbit convex polytopes and tilings}
\author[Matteo]{Kolya Matteo}
\address{Department of Mathematics, 567 Lake Hall\\Northeastern University\\Boston, MA 02115}
\email{matteo.n@husky.neu.edu}
\date{\today}

\begin{document}
\begin{abstract}
We classify the convex polytopes whose symmetry groups have two
orbits on the flags. These exist only in two or three dimensions,
and the only ones whose combinatorial automorphism group
is also two-orbit are the cuboctahedron, the icosidodecahedron,
and their duals.
The combinatorially regular two-orbit convex polytopes are certain $2n$-gons
for each $n \geq 2$.
We also classify the face-to-face
tilings of Euclidean space by convex polytopes whose symmetry groups have two flag orbits. There are finitely many families, tiling one, two, or three dimensions.
The only such tilings
which are also combinatorially two-orbit are the trihexagonal plane tiling,
the rhombille plane tiling, the tetrahedral-octahedral honeycomb,
and the rhombic dodecahedral honeycomb.
\end{abstract}
\maketitle

\section{Introduction}
Here we will classify all convex polytopes, and face-to-face tilings of Euclidean space by convex polytopes, whose flags have two orbits under the action of the symmetry group. First we briefly define these terms.

A \emph{convex polytope} is the convex hull of a finite set of points in $d$-dimensional Euclidean space $\E^d$ \cite{grunbaum1967convex}.
In this paper we use ``$d$-polytope'' to mean ``$d$-dimensional convex polytope,''
``polygon'' to mean ``2-polytope'' and ``polyhedron'' to mean ``3-polytope.''
A \emph{face} of a convex polytope $P$ is the intersection of $P$ with a \emph{supporting hyperplane} of $P$, i.e.\ a hyperplane $H$ such that $P$ is contained in one closed half-space determined by $H$, and such that $H$ and $P$ have non-empty intersection.
We also admit the empty set and $P$ itself as ``improper'' faces.
A face is called $j$-dimensional, or a \emph{$j$-face}, if its affine hull is $j$-dimensional; the empty face is $(-1)$-dimensional.
The 0-faces are also called \emph{vertices}; 1-faces are also called \emph{edges};
$(d-2)$-faces may be called \emph{ridges} and $(d-1)$-faces are called \emph{facets}.

The faces of $P$, ordered by containment, form a lattice $\sL(P)$, the \emph{face lattice} of $P$.
The \emph{symmetry group} of $P$, denoted $G(P)$, is the set of Euclidean isometries which carry $P$ to itself.
The \emph{automorphism group} of $P$, denoted $\Gamma(P)$, is the set of lattice isomorphisms from $\sL(P)$ to itself.
Since each transformation in $G(P)$ acts as an automorphism of $\sL(P)$, we can consider $G(P)$ as a subgroup of $\Gamma(P)$.

A maximal chain in $\sL(P)$ (i.e.\ a maximal linearly ordered set of faces) is called a \emph{flag} (due to the way a vertex, followed by an edge incident to that vertex, followed by a 2-face incident to the edge, resemble the construction of a flagpole.) The set of all flags of $P$ is $\sF(P)$. Transformations in $G(P)$ (or automorphisms in $\Gamma(P)$) induce an action on $\sF(P)$ in an obvious way. The orbits of flags under the action of $G(P)$ are called \emph{flag orbits}, and a polytope with $n$ distinct flag orbits is called an \emph{$n$-orbit polytope}.
Similarly,
orbits of flags under the action of $\Gamma(P)$ are called \emph{combinatorial flag orbits}, and a polytope with $n$ such orbits is called \emph{combinatorially $n$-orbit};
in the context of abstract polytopes, this is the only definition possible
and the adjectives may be dropped.

In \cite[273]{conway2008symmetries}, Conway et al. introduce the term \emph{flag rank} for the number of flag orbits. A $k$-orbit polytope is said to have a flag rank of $k$,
and they also suggest that such a polytope be called $\frac{1}{k}$-regular.
Thus, in this paper we determine all the half-regular convex polytopes.

One-orbit polytopes are the regular polytopes.
It is well known \cite{coxeter1973regular}
% an overused phrase, but I think justified here
that there are infinitely many regular polygons,
namely the regular $n$-gon for each $n \geq 3$;
there are five regular polyhedra,
the Platonic solids;
there are six regular 4-polytopes;
and there are three regular $d$-polytopes for all $d > 4$.

As far as flags are concerned, two-orbit polytopes are as close to regular as possible while not being regular.
Two-orbit convex polytopes can either be combinatorially two-orbit, if $G(P) = \Gamma(P)$,
or combinatorially regular,
in which case $G(P)$ is a subgroup of index 2 in $\Gamma(P)$.
In the more general case of abstract polytopes, 
combinatorially two-orbit polyhedra were examined by \textcite{hubard2010two}.
The \emph{chiral polytopes} are notable examples of two-orbit abstract polytopes \cite{chiral}.
However, convex polytopes cannot be chiral \cite[496]{chiral}.
%(Chiral, p. 496); chiral polytopes have a Schl\"afli symbol, hence by a theorem of McMullen [CombRegTopes,McMThesis] must be combinatorially isomorphic to a classical regular convex polytope

\begin{figure}[h]
%auto-ignore
\begin{tabular}{ccccc}
\begin{tikzpicture}[baseline=0.5cm]
\draw (0,1) -- (2,1) -- (2,0) -- (0,0) -- cycle;
\end{tikzpicture}
&&
\begin{tikzpicture}[baseline=0.8cm]
\draw (0,0) -- ++(90:.8cm) coordinate (A)
	 -- ++(90:.8cm) -- ++ (150:.4cm) coordinate (B)
	 -- ++(150:.4cm) -- ++(210:.8cm) coordinate (C)
	 -- ++(210:.8cm) -- ++(270:.4cm) coordinate (D)
	 -- ++(270:.4cm) -- ++(330:.8cm) coordinate (E)
	 -- ++(330:.8cm) -- ++(30:.4cm) coordinate (F)
	 -- ++(30:.4cm);
\end{tikzpicture}
&&
\begin{tikzpicture}[baseline=0.7cm,scale=0.95]
\draw [turtle=home,fd=1.4cm,
		left=45,fd=.8cm,
		left=45,fd=1.4cm,
		left=45,fd=.8cm,
		left=45,fd=1.4cm,
		left=45,fd=.8cm,
		left=45,fd=1.4cm,
		left=45,fd=.8cm];
\end{tikzpicture}
\\[16mm]
\begin{tikzpicture}[baseline=0.5cm,scale=1.2]
\draw (0,0.5) -- (1,1) -- (2,0.5) -- (1,0) -- cycle;
\end{tikzpicture}
&\phantom{blabla}&
\begin{tikzpicture}[baseline=0.7cm,scale=1.1]
\path (0,0) -- ++(90:.8cm) coordinate (A)
	 -- ++(90:.8cm) -- ++ (150:.4cm) coordinate (B)
	 -- ++(150:.4cm) -- ++(210:.8cm) coordinate (C)
	 -- ++(210:.8cm) -- ++(270:.4cm) coordinate (D)
	 -- ++(270:.4cm) -- ++(330:.8cm) coordinate (E)
	 -- ++(330:.8cm) -- ++(30:.4cm) coordinate (F)
	 -- ++(30:.4cm);
\draw (A) -- (B) -- (C) -- (D) -- (E) -- (F) -- cycle;
\end{tikzpicture}
&\phantom{blabla}&
\begin{tikzpicture}[baseline=0.7cm]
\draw (-2.5,.7) -- (-2.23,1.68) -- (-1.25,1.95) -- (-.28,1.68) --
		(0,.7) -- (-.28,-.28) -- (-1.25,-.55) -- (-2.23,-.28) -- cycle;
\end{tikzpicture}
\end{tabular}
\caption{The first few two-orbit convex polygons, in pairs of duals}\label{fig:polygons}
\end{figure}
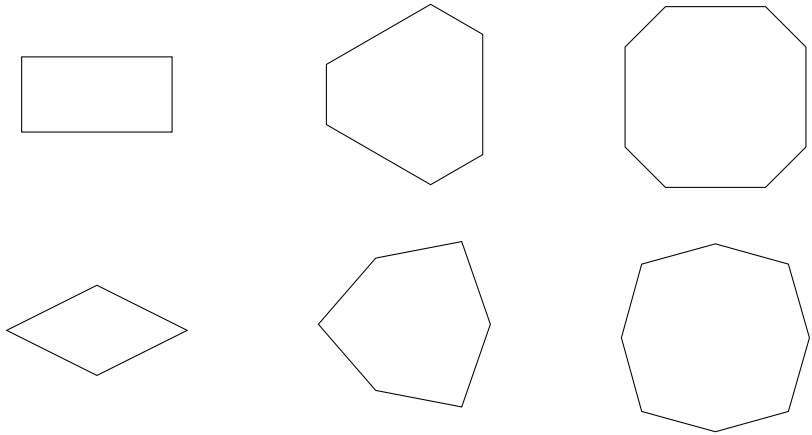

As we shall show, two-orbit convex polytopes turn out to be even scarcer than one-orbit convex polytopes, and exist only in two or three dimensions. There are infinitely many in two dimensions.
For each $n \geq 2$, a two-orbit $2n$-gon may be constructed by alternating edges of two distinct lengths, with the same angle at each vertex (namely the interior angle of a regular $2n$-gon, $\frac{n-2}{n}\pi$).
Dual to each of these is another type of two-orbit $2n$-gon, with uniform edge lengths, but alternating angle measures.

In three dimensions, there are just four:
The cuboctahedron,
its dual the rhombic dodecahedron,
the icosidodecahedron,
and its dual the rhombic triacontahedron.
We summarize the results in Theorems \ref{thm:topes} and \ref{thm:tilings}.

\begin{thm}\label{thm:topes}
There are no two-orbit $d$-polytopes if $d \geq 4$ (or if $d \leq 1$).
There are exactly four, if $d = 3$:
the cuboctahedron, icosidodecahedron, rhombic dodecahedron, and rhombic triacontahedron.
If $d = 2$, there are two infinite series of $2n$-gons,
for each $n \geq 2$. Polygons of one series alternate between two distinct edge lengths.
Polygons of the other alternate between two distinct angle measures.
\end{thm}

\begin{figure}[h]
%auto-ignore
\begin{tabular}{ccc}
\begin{tikzpicture}[scale=1.4,very thick,line join=bevel,baseline]
\coordinate (a) at (0,1,1);
\coordinate (b) at (1,0,1);
\coordinate (c) at (0,-1,1);
\coordinate (d) at (-1,0,1);
\coordinate (e) at (1,1,0);
\coordinate (f) at (1,-1,0);
\coordinate (g) at (-1,-1,0);
\coordinate (h) at (-1,1,0);
\coordinate (i) at (0,1,-1);
\coordinate (j) at (1,0,-1);
\coordinate (k) at (0,-1,-1);
\coordinate (l) at (-1,0,-1);
\draw (a) -- (b) -- (c) -- (d) -- cycle; % front
\draw (a) -- (e) -- (i) -- (h) -- cycle; % top
\draw (h) -- (d) -- (g); % left, visible
\draw[thin,gray] (g) -- (l) -- (h); % left, behind
\draw (b) -- (e) -- (j) -- (f) -- cycle; % right
\draw (f) -- (c) -- (g); % bottom, visible
\draw[thin,gray] (g) -- (k) -- (f); % bottom, behind
\draw[thin,gray] (i) -- (j) -- (k) -- (l) -- cycle; %back
\end{tikzpicture}
&&
\begin{tikzpicture}[very thick, scale=1.2,line join=bevel,baseline]
\coordinate (v1) at  ( 0,     0,     1.618034);
\coordinate (v7) at  (  .5,    .809017, 1.309017);
\coordinate (v9) at  (  .5,   -.809017, 1.309017);
\coordinate (v11) at ( -.5,    .809017, 1.309017);
\coordinate (v13) at ( -.5,   -.809017, 1.309017);
\coordinate (v23) at ( 1.309017,  .5,    .809017);
\coordinate (v25) at ( 1.309017, -.5,    .809017);
\coordinate (v27) at (-1.309017,  .5,    .809017);
\coordinate (v29) at (-1.309017, -.5,    .809017);
\coordinate (v15) at (  .809017, 1.309017,  .5);
\coordinate (v17) at (  .809017,-1.309017,  .5);
\coordinate (v19) at ( -.809017, 1.309017,  .5);
\coordinate (v21) at ( -.809017,-1.309017,  .5);
\coordinate (v3) at  ( 0,     1.618034, 0);
\coordinate (v4) at  ( 0,    -1.618034, 0);
\coordinate (v5) at  ( 1.618034, 0,     0);
\coordinate (v6) at  (-1.618034, 0,     0);
\coordinate (v16) at (  .809017, 1.309017, -.5);
\coordinate (v18) at (  .809017,-1.309017, -.5);
\coordinate (v20) at ( -.809017, 1.309017, -.5);
\coordinate (v22) at ( -.809017,-1.309017, -.5);
\coordinate (v24) at ( 1.309017,  .5,   -.809017);
\coordinate (v26) at ( 1.309017, -.5,   -.809017);
\coordinate (v28) at (-1.309017,  .5,   -.809017);
\coordinate (v30) at (-1.309017, -.5,   -.809017);
\coordinate (v8) at  (  .5,    .809017,-1.309017);
\coordinate (v10) at (  .5,   -.809017,-1.309017);
\coordinate (v12) at ( -.5,    .809017,-1.309017);
\coordinate (v14) at ( -.5,   -.809017,-1.309017);
\coordinate (v2) at  ( 0,     0,-1.618034);
\draw (v1) -- (v7) -- (v23) -- (v25) -- (v9) -- (v1);
\draw (v1) -- (v11) -- (v27) -- (v29) -- (v13) -- (v1);
\draw (v15) -- (v7) -- (v11) -- (v19) -- (v3) -- (v15);
\draw (v9) -- (v13) -- (v21) -- (v4) -- (v17) -- (v25);
\draw (v15) -- (v23) -- (v5) -- (v24) -- (v16) -- (v15);
\draw[thin,gray] (v2) -- (v8) -- (v24) -- (v26) -- (v10) -- (v2);
\draw[thin,gray] (v2) -- (v12) -- (v28) -- (v30) -- (v14) -- (v2);
\draw (v16) -- (v3) -- (v20);
\draw[thin,gray] (v20) -- (v12) -- (v8) -- (v16);
\draw (v25) -- (v5) -- (v26) -- (v18) -- (v17) -- (v9);
\draw (v21) -- (v29);
\draw[thin,gray] (v29) -- (v6) -- (v30) -- (v22) -- (v21);
\draw (v27) -- (v19) -- (v20);
\draw[thin,gray] (v20) -- (v28) -- (v6) -- (v27);
\draw[thin,gray] (v10) -- (v14) -- (v22) -- (v4) -- (v18) -- (v10);
%\foreach \x in {1,...,30}
%    \node[fill=white] at (v\x) {$\x$};
\end{tikzpicture}
\\
\vphantom{$\dfrac 1 1$}
Cuboctahedron && Icosidodecahedron
\\[1mm]
\begin{tikzpicture}[very thick,z={(-2mm,-4mm)},scale=1.1,line join=bevel,baseline]
\coordinate (a) at (1,1,1);
\coordinate (b) at (1,-1,1);
\coordinate (c) at (-1,-1,1);
\coordinate (d) at (-1,1,1);
\coordinate (e) at (1,1,-1);
\coordinate (f) at (1,-1,-1);
\coordinate (g) at (-1,-1,-1);
\coordinate (h) at (-1,1,-1);
\coordinate (i) at (0,0,2);
\coordinate (j) at (0,2,0);
\coordinate (k) at (2,0,0);
\coordinate (l) at (0,-2,0);
\coordinate (m) at (-2,0,0);
\coordinate (n) at (0,0,-2);
\draw (i) -- (a) -- (j) -- (d) -- (i);
\draw (a) -- (k) -- (b) -- (i);
\draw (i) -- (c) -- (l) -- (b);
\draw (c) -- (m) -- (d);
\draw (k) -- (e) -- (j);
\draw (e) -- (j) -- (h);
\draw[thin,gray] (h) -- (n) -- (e);
\draw[thin,gray] (k) -- (f) -- (l);
\draw[thin,gray] (l) -- (g) -- (m);
\draw (m) -- (h);
\draw[thin,gray] (g) -- (n) -- (f);
\end{tikzpicture}
&\phantom{blablablabla}&
\begin{tikzpicture}[very thick,scale=0.45,line join=bevel,baseline]
%coordinate info from http://www.rwgrayprojects.com/Lynn/Coordinates/coord01.html
% 32 verts, 60 edges, 30 faces
\coordinate (v2) at (2.618034, 0, 4.236068); 
\coordinate (v4) at (0, 1.618034, 4.236068); 
\coordinate (v6) at (-2.618034, 0, 4.236068); 
\coordinate (v8) at (0, -1.618034, 4.236068); 
\coordinate (v11) at (2.618034, 2.618034, 2.618034); 
\coordinate (v12) at (0, 4.236068, 2.618034); 
\coordinate (v13) at (-2.618034, 2.618034, 2.618034); 
\coordinate (v16) at (-2.618034, -2.618034, 2.618034); 
\coordinate (v17) at (0, -4.236068, 2.618034); 
\coordinate (v18) at (2.618034, -2.618034, 2.618034); 
\coordinate (v20) at (4.236068, 0, 1.618034); 
\coordinate (v23) at (-4.236068, 0, 1.618034); 
\coordinate (v27) at (4.236068, 2.618034, 0); 
\coordinate (v28) at (1.618034, 4.236068, 0); 
\coordinate (v30) at (-1.618034, 4.236068, 0); 
\coordinate (v31) at (-4.236068, 2.618034, 0); 
\coordinate (v33) at (-4.236068, -2.618034, 0); 
\coordinate (v34) at (-1.618034, -4.236068, 0); 
\coordinate (v36) at (1.618034, -4.236068, 0); 
\coordinate (v37) at (4.236068, -2.618034, 0); 
\coordinate (v38) at (4.236068, 0, -1.618034); 
\coordinate (v41) at (-4.236068, 0, -1.618034); 
\coordinate (v45) at (2.618034, 2.618034, -2.618034); 
\coordinate (v46) at (0, 4.236068, -2.618034); 
\coordinate (v47) at (-2.618034, 2.618034, -2.618034); 
\coordinate (v50) at (-2.618034, -2.618034, -2.618034); 
\coordinate (v51) at (0, -4.236068, -2.618034); 
\coordinate (v52) at (2.618034, -2.618034, -2.618034); 
\coordinate (v54) at (2.618034, 0, -4.236068); 
\coordinate (v56) at (0, 1.618034, -4.236068); 
\coordinate (v58) at (-2.618034, 0, -4.236068); 
\coordinate (v60) at (0, -1.618034, -4.236068);
\draw (v2) -- (v4) -- (v6) -- (v8) -- (v2)
      (v2) -- (v11) -- (v12)
      (v4) -- (v12) -- (v13) -- (v6) -- (v23)
      (v6) -- (v16) -- (v17) -- (v8)
      (v17) -- (v18)
      (v2) -- (v18)
      (v2) -- (v20) -- (v27) -- (v28)
      (v12) -- (v28)
      (v12) -- (v30)
      (v13) -- (v31)
      (v23) -- (v31)
      (v23) -- (v33)
      (v33) -- (v16)
      (v18) -- (v37) -- (v20)
      (v11) -- (v27)
      (v45) -- (v46)
      (v38) -- (v27) -- (v45)
      (v37) -- (v38)
      (v28) -- (v46)
      (v30) -- (v46)
      (v30) -- (v31)
      (v17) -- (v36)
      (v36) -- (v37);
\draw[thin,gray] (v54) -- (v45)
      (v46) -- (v56)
      (v41) -- (v58) -- (v47)
      (v52) -- (v54) -- (v38)
      (v46) -- (v47) -- (v31) -- (v41) -- (v33) -- (v50) -- (v51) -- (v52) -- (v37)
      (v54) -- (v56) -- (v58) -- (v60) -- (v54)
      (v36) -- (v51) -- (v34) -- (v17)
      (v60) -- (v51)
      (v58) -- (v50)
      (v33) -- (v34);
%\foreach \x in {2, 4, 6, 8, 11, 12, 13, 16, 17, 18, 20, 23, 27, 28, 30, 31, 33, 34, 36, 37, 38, 41, 45, 46, 47, 50, 51, 52, 54, 56, 58, 60}
%    \node[fill=white] at (v\x) {$\x$};
\end{tikzpicture}\\
\vphantom{$\dfrac 1 1$}
Rhombic Dodecahedron && Rhombic Triacontahedron
\end{tabular}
\caption{The two-orbit convex polyhedra}\label{fig:polyhedra}
\end{figure}
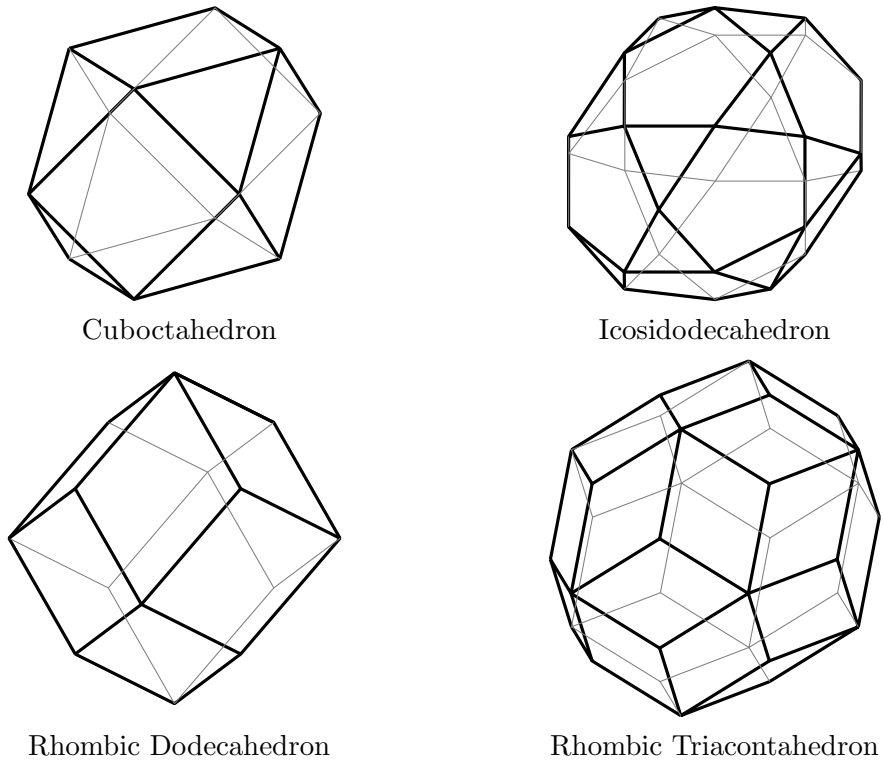

In Section~\ref{sec:tilings} we classify all two-orbit tilings by convex polytopes.
We consider only face-to-face, locally finite tilings.
See Section~\ref{sec:tilings} for a description of each of the named tilings.

\begin{thm}\label{thm:tilings}
There are no two-orbit tilings of $\E^d$ if $d \geq 4$ (or if $d = 0$).
If $d = 1$, there is one family:
an apeirogon alternating between two distinct edge lengths.
If $d = 2$, there are four:
the trihexagonal tiling (3.6.3.6);
its dual, the rhombille tiling;
a family of tilings by translations of a rhombus;
and a family of tilings by rectangles.
If $d = 3$, there are two:
the tetrahedral-octahedral honeycomb
and its dual, the rhombic dodecahedral honeycomb.
\end{thm}

In the above two theorems, all those examples which vary by a real parameter greater than one (both types of $2n$-gons, the apeirogon, and the tilings by rhombi and rectangles)
are combinatorially regular;
in each case, allowing the parameter to become one yields a regular polygon or tiling,
to which all other members of the family are isomorphic.
The other examples, namely the four polyhedra, the trihexagonal tiling,
the rhombille tiling, the tetrahedral-octahedral honeycomb, 
and the rhombic dodecahedral honeycomb,
are all unique (up to similarity),
and are all combinatorially two-orbit.

\section{Preliminary Facts}
Let $P$ be a $d$-polytope.
Recall that \emph{flags} are maximal chains of faces of~$P$.
Two flags are said to be \emph{adjacent} if they differ in exactly one face;
if they differ in the $j$-face, they are said to be $j$-adjacent.
The face lattice $\sL(P)$ satisfies the following four properties, %p. 22, Sec. 2A
which are in fact taken to be the definition of an abstract polytope of rank $d$ \cite[22]{ARP}:
\begin{enumerate}[label=(P\arabic*)]
\item
There is a least face $F_{-1}$, the \emph{empty face}, and a greatest face 
$F_d$, which is $P$ itself.
\item
Every flag contains $d+2$ faces.
\item\label{connectivity}
(Strong flag-connectivity:)
For any two flags $\Phi$ and $\Psi$ of $P$,
there exists a sequence of flags $\Phi \eqqcolon \Phi_0, \Phi_1, \dotsc, \Phi_k \coloneqq \Psi$,
such that each flag is adjacent to its neighbors
and $\Phi \cap \Psi \subseteq \Phi_i$ for each $i$.
\item
(The diamond condition:)
For any $j$, $1 \leq j \leq d$, any $j$-face $G$ of $P$,
and any $(j-2)$-face $F$ contained in $G$,
there are exactly two faces $H$ such that $F < H < G$.
\end{enumerate}
A $d$-polytope $Q$ is said to be \emph{dual} to $P$
if the face lattice $\sL(Q)$ is anti-isomorphic to the lattice $\sL(P)$,
that is, identical to $\sL(P)$ with the order reversed.
A bijective, order-reversing function $h \colon \sL(P) \to \sL(Q)$
is called a \emph{duality}.
%For convex polytopes, we require further that $Q$ have the same symmetry group as $P$.
A dual polytope to $P$ is often denoted $P^*$.
Clearly, any two duals of $P$ are combinatorially isomorphic.
A dual $P^*$ to any convex polytope $P$ may be constructed
by the process of \emph{polar reciprocation}:
After translating $P$, if necessary, so that the origin is contained in its interior,
let $P^* = \bigcap_{y \in P}\set{x}{\scprd{x}{y} \leq 1}$,
where $\scprd{x}{y}$ is the scalar product.
Then $(P^*)^* = P$ and $G(P^*) = G(P)$.
Thus, when necessary, we may assume that a polytope and its dual
have the same symmetry group.

For any two faces $F$ and $G$ of $P$ with $F \leq G$,
$G/F$ denotes the \emph{section} of $\sL(P)$
whose face lattice is $\set{H \in \sL(P)}{F \leq H \leq G}$.
This section may be realized as a convex polytope by taking the dual polytope
$G^*$ to $G$, say with a duality $h \colon G \to G^*$;
then the dual $h(F)^*$ to the face $h(F)$ of $G^*$ is the desired polytope.

A subgroup of $G(P)$ acts on the section $G/F$;
namely, those symmetries which fix all faces of $P$ which contain $G$
and all faces of $P$ which are faces of $F$.
%Namely, fix a flag $\Phi = \{F_{-1},F_0, \dotsc, F_{d-1},F_d\}$ containing the faces $F$ and $G$.
%Say $F_j = F$ and $F_k = G$.
%Consider those symmetries of $P$ which fix $F_i$ whenever $-1 \leq i \leq j$ or $k \leq i \leq d$.
%In other words, they carry $\Phi$ to other flags which differ only in the portion between $F_j$ and $F_k$.
These form a subgroup which acts faithfully on $G/F$ in a well-defined way.
%Moreover, this group does not depend on the choice of $\Phi$.
As symmetries of $G/F$, this group is a subgroup of the symmetry group of $G/F$.
We call it the \emph{restricted subgroup}, denoted $G_P(G/F)$ (this is not standard notation.)

Note that the symmetry group of a two-orbit $d$-polytope $P$ can have at most
two orbits on its $j$-faces, for any $j < d$.

\begin{clm}\label{trans}
Suppose $P$ is a two-orbit $d$-polytope.
If the symmetry group $G(P)$ is not transitive on $j$-faces
for some $j$,
then $G(P)$ is transitive on $i$-faces for all $i \neq j$,
where $0 \leq i, j \leq d-1$.
\end{clm}
\begin{proof}
%"We prove the contrapositive: If $P$ is not transitive on $i$-faces
% nor on $j$-faces, where $j < i$, then $P$ is not a two-orbit polytope."
Otherwise, we have two orbit classes of $i$-faces, say class I and II,
and two classes of $j$-faces, say A and B. Without loss of generality,
suppose $j < i$.
Let us say that a flag of $P$ whose $j$-face is in class A
and whose $i$-face is in class I is an A-I flag, and similarly
for other cases.
%%To do without using contradiction:
% Again without loss of generality,
% choose labels so that some $j$-face in class A
% is contained in an $i$-face in class I,
% and so that some $j$-face in class B
% is contained in an $i$-face in class II.
%%(We can always do this). Then we have at least A-I, B-II. skip this paragraph,
%%go straight to next, prove a third flag type exists.
Then we have more than two flag types, A-I, A-II, B-I, and B-II,
unless the $j$-faces in class A occur only in one class of $i$-faces,
say I, and $j$-faces in class B occur only in $i$-faces in class II.
But, as we will show, this violates the connectivity property \ref{connectivity}.

Let $\Phi$ be an A-I flag and $\Psi$ be a B-II flag.
By flag-connectedness there is a sequence of adjacent flags,
$\Phi = \Phi_0, \Phi_1, \dotsc, \Phi_k = \Psi$.
Let $\ell$ be the least index such that $\Phi_\ell$ contains a $j$-face
in class B or an $i$-face in class II, or both.
Then $\Phi_{\ell-1}$ is an A-I flag,
and since $\Phi_\ell$ is adjacent to $\Phi_{\ell-1}$,
only one face is different,
so $\Phi_\ell$ is either an A-II flag or a B-I flag.
Therefore $P$ has at least three flag orbits.
\end{proof}

Polytopes which are transitive on $j$-faces
for all $1 \leq j \leq d-1$ are called \emph{fully transitive}.
It is a theorem of McMullen's thesis \cite{McMThesis}
that fully transitive convex polytopes are regular.
%We reproduce it here.
%First we need another theorem from the thesis.
% We need 4A1: $P$ is combinatorially regular iff it has a Schl\"afli symbol.
%  ARP 1B9, p. 14; An equivelar convex polytope is combinatorially regular.
% 4A3: $P$ is combinatorially regular iff its facets are combinatorially regular and combinatorially equivalent and its vertex-figures are combinatorially regular and combinatorially equivalent.
% 4C1: Let $P_1$ and $P_2$ be combinatorially equivalent $d$-polytopes ($d \geq 3$)
% such that corresponding facets of $P_1$ and $P_2$ are congruent. Then $P_1$ and $P_2$
% are themselves congruent.
% 4C2: A $d$-polytope $P$ ($d \geq 3$) is regular iff its facets are regular and
% combinatorially equivalent, and its vertex-figures are combinatorially regular
% and combinatorially equivalent.
% 4C3 (with proof lifted from 4A4), and 
% 4C4 to get 4C6
%\begin{thm}[McMullen 4C4]
% A $d$-polytope $P$ ($d \geq 3$) is regular if and only if
% its vertices lie on a sphere, and its facets are regular and combinatorially equivalent.
%\end{thm}
\begin{thm}[McMullen {\cite[\nopp 4C6]{McMThesis}}]\label{fullytransreg}
A $d$-polytope $P$ is regular if and only if
for each $j = 0, \dotsc, d-1,$ its symmetry group $G(P)$
is transitive on the $j$-faces of $P$.
\end{thm}

Therefore, for a two-orbit $d$-polytope $P$
there is a $j$, $0 \leq j \leq d-1$,
so that $G(P)$
is not transitive on the $j$-faces
but is transitive
on the faces of every other rank.
We shall call such a polytope $j$-intransitive.
%The $j$-faces of a $j$-intransitive two-orbit polytope
%fall in two distinct orbits.

In the language of \textcite{hubard2010two}, a 0-intransitive two-orbit polyhedron
is of class $2_{1,2}$, a 1-intransitive two-orbit polyhedron is of class $2_{0,2}$,
and a 2-intransitive two-orbit polyhedron is of class $2_{0,1}$.
Claim~\ref{trans} and the above comments were proved in \cite{hubard2010two}.
They are consequences of Theorem~5 therein,
which we may paraphrase to say
that an (abstract) two-orbit $d$-polytope $P$ is either fully transitive,
or there exists a $j$ ($1 \leq j \leq d$) such that $P$
is $i$-transitive for every $i \neq j$, but not for $i = j$.
In using any results about abstract two-orbit polytopes, however,
we must be careful to remember that convex two-orbit polytopes
may be combinatorially regular and not combinatorially two-orbit.

\begin{clm}\label{freeflags}
For any convex polytope $P$,
the order of the symmetry group $G(P)$ divides the number of flags of $P$.
Each flag orbit has the same size, namely $\abs{G(P)}$,
and so $P$ is a two-orbit polytope if and only if the number of flags
is twice the order of $G(P)$.
\end{clm}
\begin{proof}
This all follows from the fact that $G(P)$ acts freely
on the set of flags of $P$.
Let $\Phi$ be any flag of $P$.
Any $\gamma \in G(P)$
acts on the $j$-adjacent flag $\Phi^j$ to $\Phi$
as $\gamma(\Phi^j) = \gamma(\Phi)^j$,
since $\gamma$ is an automorphism of the face lattice.
Therefore, if $\gamma \in G(P)$ is such that
$\gamma(\Phi) = \Phi$,
then $\gamma$ will also fix each flag adjacent to $\Phi$,
and thus all flags of $P$ by flag-connectedness,
so $\gamma$ is the identity.
\end{proof}
It follows that the dual to a two-orbit polytope 
is two-orbit; the dual to a $j$-intransitive $d$-polytope
is $(d-j-1)$-intransitive.

\begin{clm}\label{adjflags}
If $P$ is a two-orbit $j$-intransitive $d$-polytope,
and $\Phi$ is any flag,
then for any $i \neq j$
the $i$-adjacent flag $\Phi^i$
is in the same orbit as $\Phi$.
That is, there exists a symmetry $\rho \in G(P)$
such that $\rho(\Phi) = \Phi^i$.
\end{clm}
\begin{proof}
Since there are only two flag orbits,
and two classes of $j$-faces,
the orbit of a given flag is determined entirely by
its $j$-face.
For $i \neq j$,
$\Phi$ and $\Phi^i$ share their $j$-face,
hence are in the same flag orbit.
\end{proof}

\begin{crl}\label{inadjflags}
If $P$ is a two-orbit $j$-intransitive $d$-polytope,
and $\Phi$ is any flag,
then the $j$-adjacent flag $\Phi^j$
is not in the same flag orbit as $\Phi$.
\end{crl}
\begin{proof}
If $\Phi^j$ were in the same orbit as $\Phi$,
then by Claim~\ref{adjflags},
for each $i = 0, \dotsc, d - 1$ there exists
an isometry $\rho_i$ of $P$
such that $\rho_i(\Phi) = \Phi^i$.
But if a flag is in the same orbit 
as all of its adjacent flags,
it follows from flag-connectedness that $P$ is regular
(see Proposition 2B4 of \cite{ARP} or Theorem 4B1 of \cite{McMThesis}.)
\end{proof}

The next corollary is immediate from Corollary~\ref{inadjflags}.

\begin{crl}\label{alternate}
If $P$ is a two-orbit $j$-intransitive $d$-polytope,
then for any $(j+1)$-face $F_{j+1}$ of $P$
and any $(j-1)$-face $F_{j-1}$ contained in $F_{j+1}$,
the two $j$-faces $H$ with $F_{j-1} < H < F_{j+1}$
are in different $j$-face orbits.
\end{crl}

In the following, by ``chain of cotype $\{j\}$'' we mean
a chain of faces in $\sL(P)$
including a face of each rank except $j$.

\begin{clm}\label{cochain-trans}
If $P$ is a two-orbit $j$-intransitive $d$-polytope,
then $G(P)$ acts transitively on chains of cotype $\{j\}$.
\end{clm}
\begin{proof}
Let $\Psi$ and $\Omega$ be two chains of cotype $\{j\}$.
By Corollary~\ref{alternate},
the two $j$-faces which are incident to the $(j-1)$-face
and $(j+1)$-face of $\Psi$ are in different $j$-face orbits.
Recall that the orbit of a given flag is determined entirely by
its $j$-face.
So we may extend $\Psi$ to a flag in either flag orbit.
Similarly, we may extend $\Omega$ to a flag in either orbit.
Thus, we extend $\Psi$ to a flag $\Psi'$
and $\Omega$ to a flag $\Omega'$ such that both are in the same orbit;
then there is a symmetry $\gamma \in G(P)$
so $\gamma(\Psi') = \Omega'$,
and thus $\gamma(\Psi) = \Omega$.
\end{proof}

\begin{clm}\label{vertexorfacet}
If $P$ is a two-orbit $j$-intransitive $d$-polytope,
then $j = 0$ or $j = d-1$.
\end{clm}
\begin{proof}
Suppose $1 \leq j \leq d - 2$.
Then there is a $(j-2)$-face $F_{j-2}$ contained
in some $(j+2)$-face $F_{j+2}$ in $P$.
The section $Q = F_{j+2}/F_{j-2}$
is a polyhedron.
By Claim~\ref{cochain-trans},
isometries in the restricted group $G_P(Q)$
act transitively on the vertices and facets of $Q$
(corresponding to $(j-1)$-faces and $(j+1)$-faces of $P$,
respectively.)
By vertex transitivity,
every vertex is in the same number $q$ of edges.
By Corollary~\ref{alternate},
the edge orbits alternate across each facet,
so $q$ is even.
By facet transitivity,
each facet is a $p$-gon for some $p$,
and again by Corollary~\ref{alternate}
the edge orbits alternate at each vertex,
so $p$ is even.

%In this situation, $Q$ is said to possess
%a Schl{\"a}fli symbol $\{p,q\}$.
%By Theorems~3A1 and 4A1
%of \textcite{McMThesis}, $P$ is combinatorially isomorphic to a regular polyhedron.
%But no Platonic solid has both even-sided faces
%and an even number of edges at each vertex.
However, this contradicts Euler's theorem.
In fact, each polyhedron without triangular facets has
at least one 3-valent vertex \cite[237]{grunbaum1967convex}.
\end{proof}

\begin{clm}\label{regfaces}
If $P$ is a two-orbit $j$-intransitive $d$-polytope,
then all $i$-faces, for $i \leq j$, are regular.
% Any $i$-face $F$ with $i > j$ has two orbits under the stabilizer
% of $F$ in $G(P)$.
More generally, any section $G/F$, where $G$ is a $k$-face
and $F$ is an $l$-face, is regular if $j \leq l$ or $k \leq j$.
If $l < j < k$, then $G/F$ has two flag orbits under the restricted subgroup 
$G_P(G/F)$.
\end{clm}
\begin{proof}
Since there are only two flag orbits,
and two classes of $j$-faces,
the orbit of a given flag is determined entirely by
its $j$-face.
Suppose $G/F$ is a section as described
and we do not have $l < j < k$.
Choose a base flag $\Phi$ of $G/F$
and extend it to a flag $\Phi'$ of $P$.
Now any flag $\Psi$ of $G/F$ may be extended
to a flag $\Psi'$ of $P$ which agrees with $\Phi'$
for all $i$-faces with $i \leq l$ or $i \geq k$.
In particular, $\Phi'$ and $\Psi'$ share the same $j$-face,
so there is an isometry $\gamma \in G(P)$
such that $\gamma(\Phi') = \Psi'$.
Then $\gamma$ restricts to a symmetry of $G/F$
carrying $\Phi$
to $\Psi$.
Hence $G/F$ is regular.

On the other hand, if $l < j < k$,
then $G/F$ contains a $(j-1)$-face $F_{j-1}$ of $P$
and a $(j+1)$-face $F_{j+1}$ of $P$
which contains $F_{j-1}$.
By Corollary~\ref{alternate},
the two $j$-faces $H$ of $P$
with $F_{j-1} < H < F_{j+1}$
are in different orbits.
Thus $G/F$ has at least two flag orbits under those isometries in $G(P)$
which restrict to $G/F$.
On the other hand, for any two flags $\Phi$ and $\Psi$ of $G/F$ which contain 
the same kind of $j$-face of $P$,
we may extend these to flags $\Phi'$ and $\Psi'$ of $P$
which agree on all $i$-faces with $i \leq l$ and $i \geq k$.
Then an isometry $\gamma \in G(P)$
exists with $\gamma(\Phi') = (\Psi')$,
and this $\gamma$ restricts to $G/F$
where it takes $\Phi$ to $\Psi$.
Hence $G/F$ has two flag orbits
under those transformations in $G(P)$ which restrict to $G/F$.
\end{proof}

Note that those sections in Claim~\ref{regfaces}
with two flag orbits under the restricted subgroup
are either two-orbit polytopes or regular. Their full group of symmetries
includes the restricted subgroup, but may be bigger.
If the section is in fact two-orbit, then its symmetry group
agrees with the restricted subgroup. In particular,
if a face $F$ of a two-orbit $j$-intransitive polytope
is two-orbit, then $F$ is also $j$-intransitive;
note than then $j = 0$, by Claim~\ref{vertexorfacet}.

\section{Two Dimensions}
Suppose $P$ is a two-orbit polygon.
If $P$ does not have all edges of the same length,
then it has two distinct edge lengths;
if it had three or more, then there would be three or more flag orbits.
In this case, $P$ is not edge-transitive, so it must be vertex-transitive.
Then no two edges of the same length may be adjacent,
since in that case, by vertex-transitivity, all edges would be the same length.
So $P$ must alternate edges of two distinct lengths,
and by vertex-transitivity all angles are the same.

On the other hand, suppose $P$ does have all edges the same length.
If the angle at each vertex is the same,
then $P$ would be regular.
Therefore, $P$ has at least two distinct angles;
it has at most two, since there at most two vertex orbits.
Then $P$ is not vertex-transitive,
so it must be edge-transitive,
which implies that $P$ alternates between two distinct angles.

We have shown that every two-orbit convex polygon must be of one of the two types described above. It is not hard to see that, moreover, such $2n$-gons exist for each $n \geq 2$. The existence of non-regular rectangles is well known. For each $n \geq 3$, a polygon of the first type may be constructed from a regular $n$-gon by truncation, i.e.\ chopping off a corner at each vertex. In the top row of Figure~\ref{fig:polygons}, you may see how the hexagon is a truncated equilateral triangle, and the octagon is a truncated square.

The existence of each $2n$-gon of the second type is then clear, since they are the duals of the polygons of the first type; i.e.\ they may be constructed by taking the convex hull of vertices placed at the midpoint of each edge of a polygon of the first type.

It is also clear that such polygons are, indeed, two-orbit.
Let us consider a polygon $P$ of the first type.
It then follows for the second type by duality.
Since $P$ is not edge-transitive, it has at least two flag orbits.
Since $P$ is a truncated regular $n$-gon, it has (at least) all the symmetries of the regular $n$-gon, which has order $2n$.
But $P$ has $4n$ flags ($2n$ vertices, each in 2 edges),
so $P$ has at most $4n/2n = 2$ flag orbits.
Therefore $P$ is a two-orbit polygon.

\section{Three Dimensions}
A \emph{quasiregular} polyhedron is vertex-transitive and has exactly two kinds of facets, which are regular and alternate around each vertex. By Claims~\ref{trans} and \ref{regfaces},
any 2-intransitive two-orbit polyhedron is vertex-transitive, edge-transitive,
and has regular facets in two orbits.
The two types of facet must alternate around each vertex,
i.e.\ each edge must be incident to one facet of each type,
by edge-transitivity.
Thus any 2-intransitive two-orbit polyhedron is quasiregular.
But there are only two quasiregular polyhedra: the cuboctahedron
and the icosidodecahedron, two of the Archimedean solids \cite[18]{coxeter1973regular}.

We may verify that these are two-orbit polyhedra. The cuboctahedron has at least two flag orbits, since it is not regular, having both square and triangular faces.
It has 12 vertices, each incident to 4 edges, and each edge is in 2 faces,
so it has $12 \cdot 4 \cdot 2 = 96$ flags.
The cuboctahedron may be formed by truncating each vertex of the 3-cube
at the midpoints of the edges,
so it retains all the symmetries of the cube,
a group of order 48.
Hence the cuboctahedron has at most $96/48 = 2$ orbits,
and thus is a two-orbit polyhedron (and also combinatorially two-orbit.)

The icosidodecahedron has at least two flag orbits,
since it is not regular, having both triangular and pentagonal faces.
It has 30 vertices, each in 4 edges, and each edge is in 2 faces,
so it has $30 \cdot 4 \cdot 2 = 240$ flags.
The icosidodecahedron may be formed by truncating each vertex of the
dodecahedron at the midpoints of the edges,
so it retains all the symmetries of the dodecahedron,
a group of order 120.
Hence the icosidodecahedron has at most $240/120 = 2$ orbits,
and thus is a two-orbit polyhedron (and also combinatorially two-orbit.)

Any two-orbit polyhedron which is 0-intransitive must be dual
to one of these two, so we have the rhombic dodecahedron, dual to the cuboctahedron, and the rhombic triacontahedron, dual to the icosidodecahedron. As duals to Archimedean solids, these are Catalan solids.

Rather than using the list of quasiregular polyhedra,
it is possible to arrive at candidates for 0-intransitive or 2-intransitive two-orbit polyhedra
by considering all the edge-transitive polyhedra. It turns out there are only nine: the five platonic solids, the cuboctahedron, the icosidodecahedron, the rhombic dodecahedron, and the rhombic triacontahedron \cite{graver1997locally,grunbaum1987edge}.

By Claim~\ref{vertexorfacet},
there are no 1-intransitive two-orbit polyhedra.
In fact, polyhedra which are vertex-transitive and facet-transitive have a name,
the \emph{noble} polyhedra,
and the only non-regular ones (i.e. the 1-intransitive polyhedra)
are disphenoid tetrahedra,
which are tetrahedra with non-equilateral triangular faces %\cites[475]{grunbaum2003your}[sec. 6]{hedrahollow}
\cite[26]{bruckner1906}.
It is not hard to see that, if not regular,
a tetrahedron has at least three flag orbits.

Hence the cuboctahedron, icosidodecahedron, rhombic dodecahedron, and rhombic triacontahedron are the only two-orbit polyhedra.
The same result is found in \textcite[427]{orbanic2010map} as a consequence of Theorem 6.1 therein,
stating that every 2-orbit map on the sphere is either the medial of a regular map on the sphere, or dual to one.

\section{Higher dimensions}
Suppose $P$ is a $j$-intransitive two-orbit $d$-polytope with $d \geq 4$;
by Claim~\ref{vertexorfacet} $j$ is either $0$ or $d-1$.
Any two-orbit 0-intransitive polytope is dual to a two-orbit $(d-1)$-intransitive polytope,
so we shall restrict our attention to the latter case.
Such a polytope is vertex-transitive, and by Claim~\ref{regfaces} has regular facets.
This is the definition used by \textcite{gosset1900regular} for \emph{semiregular} 
polytopes.
In his 1900 paper he gives a complete list of all the semiregular polytopes.
The list was proved to be complete in \textcite{Blind1991semireg}.

There are only seven semiregular convex polytopes in dimensions greater than three.
There are three 4-polytopes: the rectified 4-simplex, the snub 24-cell, and the rectified 600-cell.
The rectified 4-simplex, which Gosset called ``tetroctahedric,''
is the convex hull of the midpoints of the edges of the 4-simplex.
The facets are tetrahedra and octahedra.
It has 360 flags,
with 10 vertices, each in 6 edges, each edge in 3 ridges, and each ridge in 2 facets.
It has the same symmetry group as the 4-simplex, of order 120; hence it has three flag orbits.

The rectified 600-cell, which Gosset called ``octicosahedric,'' 
is the convex hull of the midpoints of the edges of the 600-cell.
The facets are octahedra and icosahedra.
It has 43,200 flags,
with 720 vertices, each in 10 edges, each edge in 3 ridges and each ridge in 2 facets.
It has the same symmetry group as the 600-cell, of order 14,400;
hence it has three flag orbits.

The snub 24-cell, which Gosset called ``tetricosahedric,'' has icosahedra
and tetrahedra for facets.
%\cite[266]{johnson1966theory}
It has 96 vertices, each in 9 edges;
6 of these edges are in 3 ridges,
and the other 3 edges are in 4 ridges.
(This already makes it clear that there are at least two orbit classes of edges,
as well as at least two orbit classes of facets,
so it cannot be two-orbit.)
Each ridge is in 2 facets.
Hence there are 5,760 flags.
It has half the symmetries of the 24-cell, leaving 576.
So it has ten flag orbits.

The remaining examples form Coxeter's $k_{21}$ family \cites[\S 11.8]{coxeter1973regular}{coxeter1988regular}, %[\S 8.4]{johnson1966theory},
with one each in dimensions 5 through 8.
They are the 5-demicube, or $1_{21}$, Gosset's ``5-ic Semi-regular'';
$2_{21}$ or ``6-ic Semi-regular''; $3_{21}$ or ``7-ic Semi-regular''; and $4_{21}$ or ``8-ic Semi-regular''.
Each of these has the preceding one for its vertex figure,
starting with the rectified 4-simplex (which may also be called $0_{21}$)
as the vertex figure of the 5-demicube.
Of course, by Claim~\ref{regfaces},
if any member of this family were two-orbit,
then the previous member (being a section) would either be two-orbit or regular.
So by induction, none of these polytopes are two-orbit.
In fact, each has three flag orbits.

Thus, no two-orbit convex polytopes exist in more than three dimensions.

In \cite[409--411]{conway2008symmetries}, Conway et al. say
that the
$n$-dimensional demicube, i.e.\ the convex hull of alternate vertices of
the $n$-cube (which they call a hemicube), has $n-2$ flag orbits. So
the 4-demicube should be two-orbit. The
4-demicube is described specifically as a 4-crosspolytope ``but with only half its symmetry.''
This apparently contradicts our result!

%However, the facets are all regular tetrahedra;
%some formed as 3-demicubes within the cubic facets of the 4-cube,
%and others as ``vertex figures'' around the deleted vertices.
%Then it is easy to establish that the 4-demicube is in fact regular,
%for instance by Theorem~4C3 of \cite{McMThesis}:
%
%%\begin{thm4c3}
%A $d$-polytope ($d \geq 4$) is regular if and only if its facets are regular
%and its vertex-figures are combinatorially regular.
%%\end{thm4c3}

However, if the 4-cube has for its vertices
the 16 points in $\E^4$ with all coordinates 0 or 1,
then the vertices of the 4-demicube are $(0,0,0,0)$,
$(1,1,1,1)$, and all vectors with two 0's and two 1's.
Hence if $x$ is a vertex, so is $\1 - x$, where $\1 = (1,1,1,1)$.
Grouping the 8 vertices in pairs $(x, \1 - x)$,
we find four axes
which are mutually perpendicular.
Thus we have four antipodal pairs of vertices of a regular
4-crosspolytope.
Hence the ``two-orbit'' 4-demicube is actually a regular 4-crosspolytope
with artificially restricted symmetries, essentially by coloring the
facets depending whether they were formed inside a facet, or at a missing
vertex, of the 4-cube.

\section{Tilings}\label{sec:tilings}
A \emph{tiling} of $d$-dimensional Euclidean space $\E^d$, also called a tessellation or a honeycomb,
is a countable collection of subsets (called \emph{tiles}) of $\E^d$ which cover $\E^d$ without gaps or overlaps;
that is, the union of the tiles is $\E^d$, and the interiors of the tiles are pairwise disjoint.
Here, we consider only locally finite face-to-face tilings by convex polytopes,
meaning that all the tiles must be convex polytopes,
every compact subset of $\E^d$ meets only finitely many tiles,
and the intersection of any two tiles is a face of both (possibly the empty face).
The face lattice of a tiling of $d$-dimensional space meets all the criteria defining
an abstract polytope of rank $d+1$,
and we call it a rank $(d+1)$ tiling.
The $d$-dimensional tiles are the facets.
A rank 3 tiling is called a \emph{plane tiling},
and a rank 2 tiling is called an \emph{apeirogon}. The latter
necessarily consists of infinitely many edges (line segments) covering the line,
and has been described as the limit of a sequence of $n$-gons as $n \to \infty$.

A \emph{normal} tiling has
\begin{itemize}
\item
tiles which are homeomorphic to closed balls,
\item
two positive radii $r$ and $R$ such that every tile contains a ball of radius $r$
and is contained in a ball of radius $R$,
and
\item
 the property that the intersection of any two tiles is empty or connected.
\end{itemize}
A two-orbit tiling has at most two congruence classes of tiles,
so that the tiles are uniformly bounded (above and below) by balls of
two given radii; together with convex polytopes as tiles, this is sufficient
to establish that the tiling is normal. This rules out certain pathological possibilities for tilings.

Claim~\ref{trans} still applies: if a two-orbit tiling is not fully transitive,
then it is not transitive on the faces of exactly one dimension, say $j$,
and we call it $j$-intransitive.
However, Theorem~\ref{fullytransreg} does not apply;
the proof depends on the fact that the vertices of a vertex-transitive polytope
lie on a sphere,
which is not the case for a tiling.
So fully transitive two-orbit tilings are a possibility (and some exist.)
Claim~\ref{freeflags} no longer makes sense,
since the symmetry group and the set of flags are both infinite,
but Claim~\ref{adjflags} and its corollaries still hold
for any $j$-intransitive two-orbit tilings.
Finally, Claim~\ref{regfaces} applies:
the faces and sections of a two-orbit tiling have at most two orbits.

Following \cite{grunbaum1986tilings}, we say two tilings are equal
if one can be mapped onto the other by a uniform scale transformation
followed by an isometry.

\subsection{Apeirogons}
There is one two-orbit tiling of the line, which varies by a single real
parameter greater than one:
an apeirogon alternating between two distinct edge lengths. 
Note that the construction of well-behaved duals does not work,
in general, for tilings, as it does for polytopes.
For example, if one constructs a ``dual'' to this two-orbit apeirogon
by taking edge midpoints for vertices, one obtains a regular
apeirogon, which is then self-dual!

This tiling is combinatorially regular.

\subsection{Plane tilings}
We consider four cases of plane tilings, based on their transitivity properties.

\subsubsection{Fully transitive}
%By facet-transitivity,
%each tile of a fully transitive plane tiling $\sT$ is congruent to every other; say they are $p$-gons.
%By vertex-transitivity, each vertex is incident to the same number, say $q$, of edges;
%hence $\sT$ has a Schl\"afli symbol $\{p,q\}$.
%The only Schl\"afli symbols for tilings are $\{3,6\}$, $\{6,3\}$, and $\{4,4\}$.
\textcite{grunbaum1986tilings} contains the full list of isohedral (i.e.\ tile-transitive) plane tilings (Table 6.1), isotoxal (i.e.\ edge-transitive) plane tilings (Table 6.4), and isogonal (vertex-transitive) plane tilings (Table 6.3). There are only four plane tilings realizable by convex tiles which have all three properties:
the three regular plane tilings
and a tiling by translations of a rhombus, labeled IH74 as an isohedral tiling,
IG74 as an isogonal tiling, and IT20 as an isotoxal tiling.
On \cite[311]{grunbaum1986tilings} it is confirmed that this rhombus tiling
is the only non-regular fully transitive tiling realizable by convex tiles.
Figure~\ref{fig:rhombus} shows a portion of this tiling, with flags of one orbit shaded.
For a given flag $\Phi$, both the 0-adjacent flag $\Phi^0$ and the 2-adjacent flag $\Phi^2$ are in the other orbit, whereas the 1-adjacent flag $\Phi^1$ remains in the same orbit;
thus with the notation of \textcite{hubard2010two} this tiling is in class $2_{1}$.

\begin{figure}[h]
\begin{tikzpicture}[xslant=.5,yscale=0.894]
\foreach \x in {1,2,3,4} {
 \foreach \y in {1,2,3} {
  \fill[gray!50] (\x,\y) -- (\x,\y+.5) -- (\x+.5,\y+.5) -- cycle;
  \fill[gray!50] (\x,\y) -- (\x+.5,\y) -- (\x+.5,\y+.5) -- cycle;
  \fill[gray!50] (\x+1,\y+1) -- (\x+.5,\y+1) -- (\x+.5,\y+.5) -- cycle;
  \fill[gray!50] (\x+1,\y+1) -- (\x+1,\y+.5) -- (\x+.5,\y+.5) -- cycle;
  \draw[white](\x,\y) -- (\x+1,\y+1);
 }
}
\draw[very thick] (.5,.5) grid (5.5,4.5);
\end{tikzpicture}
\caption{The fully-transitive rhombus tiling}\label{fig:rhombus}
\end{figure}
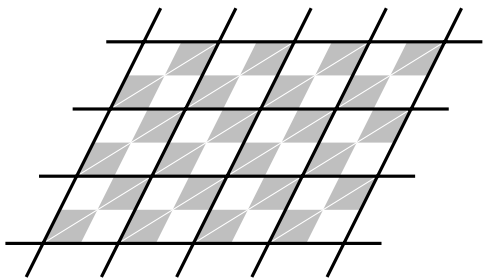

A family of unequal versions of this tiling may be obtained by varying a single real parameter greater than one (the ratio of the diagonals of the rhombus.)
The tiling is self-dual when taking tile midpoints for vertices. It is combinatorially regular.

\subsubsection{2-intransitive}
The facets of a 2-intransitive two-orbit tiling must be regular,
by Claim~\ref{regfaces}.
By edge-transitivity,
the two facets bordering each edge are from different orbits;
hence they alternate around each vertex.
By vertex-transitivity,
each vertex appears in the same kinds of tiles,
which appear in the same order around each vertex;
a common notation for such a situation
is $(p.q.r\ldots)$ to indicate that each vertex $v$ is in a $p$-gon
adjacent to a $q$-gon (containing $v$) adjacent to an $r$-gon, etc.
An exponent may be used to indicate repetition;
for instance, the regular tiling by equilateral triangles, $(3.3.3.3.3.3)$,
is denoted $(3^6)$.

If six facets appear at each vertex, then they must all be triangles,
since replacing any triangle by a regular $n$-gon with $n \geq 4$
will not fit in the plane. The only tiling with six equilateral triangles at every vertex is the regular tiling $(3^6)$.
Hence there must be exactly four facets at each vertex.

If none of the facets are triangles,
then each has at least four sides.
Four squares fit exactly around a vertex,
but replacing any squares by regular $n$-gons with $n \geq 5$
will not fit in the plane.
The only tiling with four squares at every vertex is the regular tiling $(4^4)$.
Hence there must be at least some triangles.

If all four faces at each vertex are equilateral triangles,
there is too much angular deficiency to tile the plane;
indeed, the only such figure is the regular octahedron, $(3^4)$.

If triangles alternate with squares,
the resulting figure is the cuboctahedron, $(3.4.3.4)$.
If triangles alternate with pentagons,
the resulting figure is the icosidodecahedron, $(3.5.3.5)$.
(This is, in brief, the proof that these are the only quasiregular polyhedra.)

If triangles alternate with hexagons,
we do obtain a plane tiling, denoted $(3.6.3.6)$.
This is one of the 11 \emph{uniform} plane tilings,
also called \emph{Archimedean} tilings.
This tiling, seen in Figure~\ref{fig:trihex}, is sometimes called ``trihexagonal'' or ``hexadeltille.''

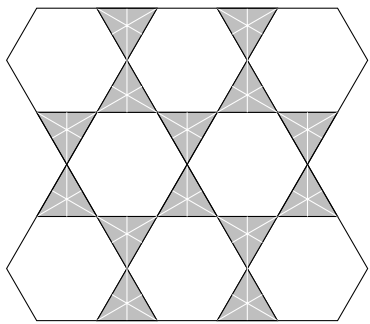
\begin{figure}[h]
\begin{tikzpicture}[scale=0.8]
\foreach \x in {0,2,4} {
  \foreach \y in {0,3.464} {
    \draw (\x,\y) -- (\x+1,\y) -- (\x+1.5,\y+.866) -- (\x+1,\y+1.732) -- (\x,\y+1.732) -- (\x-.5,\y+.866) -- cycle;
  }
  \foreach \y in {1.732} {
    \filldraw[fill=gray!50] (\x,\y) -- (\x+1,\y) -- (\x+.5,\y+.866) -- cycle;
    \filldraw[fill=gray!50] (\x,\y+1.732) -- (\x+1,\y+1.732) -- (\x+.5,\y+.866) -- cycle;
    \draw[white] (\x,\y) -- +(30:.866)
                 (\x+1,\y) -- +(150:.866)
                 (\x+.5,\y) -- (\x+.5,\y+1.732)
                 (\x,\y+1.732) -- +(-30:.866)
                 (\x+1,\y+1.732) -- +(210:.866);
  }
}
\foreach \x in {1,3} {
  \foreach \y in {1.732} {
    \draw (\x,\y) -- (\x+1,\y) -- (\x+1.5,\y+.866) -- (\x+1,\y+1.732) -- (\x,\y+1.732) -- (\x-.5,\y+.866) -- cycle;
  }
  \foreach \y in {0,3.464} {
    \filldraw[fill=gray!50] (\x,\y) -- (\x+1,\y) -- (\x+.5,\y+.866) -- cycle;
    \filldraw[fill=gray!50] (\x,\y+1.732) -- (\x+1,\y+1.732) -- (\x+.5,\y+.866) -- cycle;
    \draw[white] (\x,\y) -- +(30:.866)
                 (\x+1,\y) -- +(150:.866)
                 (\x+.5,\y) -- (\x+.5,\y+1.732)
                 (\x,\y+1.732) -- +(-30:.866)
                 (\x+1,\y+1.732) -- +(210:.866);
  }
}
\draw (1,5.196) -- (2,5.196)  (3,5.196) -- (4,5.196)
      (1,0) -- (2,0)  (3,0) -- (4,0)
      (0,3.464) -- (.5,2.598) -- (0,1.732)
      (5,3.464) -- (4.5,2.598) -- (5,1.732);
\end{tikzpicture}
\caption{The trihexagonal tiling}\label{fig:trihex}
\end{figure}

If we replace the hexagons by regular $n$-gons with $n \geq 7$,
the total angles are excessive to fit in the plane.
Hence $(3.6.3.6)$ is the unique two-orbit 2-intransitive plane tiling.
\textcite[60]{coxeter1973regular} calls it by the extended Schl\"afli symbol 
$\vsch{3}{6}$, which is suggestive
of the construction by taking the midpoints of the edges
of the regular tiling $\{3,6\}$,
or equivalently of its dual, the regular tiling $\{6,3\}$.
He describes it as a quasiregular tessellation.

It is combinatorially two-orbit.
Taking the dual by using tile midpoints for vertices works well
and results in the rhombille tiling detailed below.

\subsubsection{1-intransitive}
By facet-transitivity,
each facet has the same number of sides, say $p$,
and by vertex-transitivity,
each vertex is incident to the same number of edges, say $q$.
Thus a 1-intransitive plane tiling has a Schl\"afli symbol $\{p,q\}$.
Since edges of the two orbits alternate
at each vertex of a tile,
$p$ and $q$ are both even;
the only possible symbol is $\{4,4\}$.
The tiles must be regular or two-orbit.
The only tiling by squares is regular;
so the tiles must be two-orbit 4-gons, i.e.\ rectangles or rhombi.

It follows from vertex-transitivity, or from adding angle defects,
that rhombi must be arranged with two acute angles and two obtuse angles
at each vertex.
In the case that the two angle types alternate, we obtain the tiling in
Figure~\ref{fig:rhombus}, which we know to be fully transitive.
In the case that the obtuse angles are adjacent to each other,
and the acute angles are adjacent to each other,
we do obtain a 1-intransitive plane tiling.
The rhombi are arranged in strips which alternate direction.
However, this tiling actually has four orbits.
Indeed, in a 1-intransitive two-orbit tiling,
the orbit of a flag is determined entirely by the edge it contains;
if any face is also two-orbit,
so that its symmetry group is the same as the restricted subgroup,
then its flag orbits must also be determined by edges,
and not vertices as in the case of a rhombus.
%(In another sign of the inadequacy of tiling duals,
%if you try to construct a dual
%using tile midpoints for vertices, you get a self-dual rectangular
%tiling!)

This leaves only the tiling by copies of a rectangle. 
This is the unique two-orbit 1-intransitive family of plane tilings,
and varies by a single real parameter greater than one.
It is self-dual
and combinatorially regular,
being isomorphic to the square tiling $(4^4)$.
Figure~\ref{fig:rect} shows a patch of this tiling, with flags of one orbit shaded.

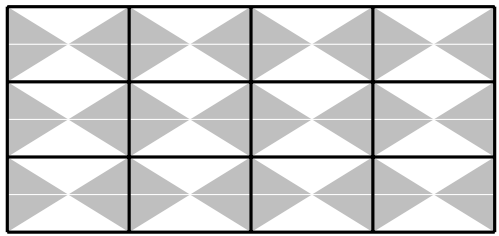
\begin{figure}[h]
\begin{tikzpicture}[xscale=1.62]
\foreach \x in {1,2,3,4} {
 \foreach \y in {1,2,3} {
  \fill[gray!50] (\x,\y) -- (\x,\y+.5) -- (\x+.5,\y+.5) -- cycle;
  \fill[gray!50] (\x,\y+1) -- (\x,\y+.5) -- (\x+.5,\y+.5) -- cycle;
  \fill[gray!50] (\x+1,\y) -- (\x+1,\y+.5) -- (\x+.5,\y+.5) -- cycle;
  \fill[gray!50] (\x+1,\y+1) -- (\x+1,\y+.5) -- (\x+.5,\y+.5) -- cycle;
  \draw[white](\x,\y+.5) -- (\x+1,\y+.5);
 }
}
\draw[very thick] (1,1) grid (5,4);
\end{tikzpicture}
\caption{The 1-intransitive rectangle tiling}\label{fig:rect}
\end{figure}

\subsubsection{0-intransitive}
It is tempting to say that any 0-intransitive tiling
must be dual to a 2-intransitive one.
However, \textcite{grunbaum1986tilings} admonish us
that for tilings, no duality theorem exists
which would allow us to make such statements!
Nonetheless, it turns out that the only 0-intransitive two-orbit tiling
is indeed dual to the uniform tiling $(3.6.3.6)$.
We can confirm this by
again turning to the tables of isohedral 
and isotoxal tilings in \cite{grunbaum1986tilings};
the only additional tiling realizable by convex tiles with both properties
is denoted IH37 as an isohedral tiling
and IT11 as an isotoxal tiling.

This is a tiling by copies of a rhombus, which can be
viewed as dividing the hexagons of the regular tiling $(6^3)$ into three rhombi
each. It is called ``rhombille'' or ``tumbling blocks,'' and is familiar
as the visual illusion of a stair-case of blocks which can be seen in two ways.
It is combinatorially two-orbit.

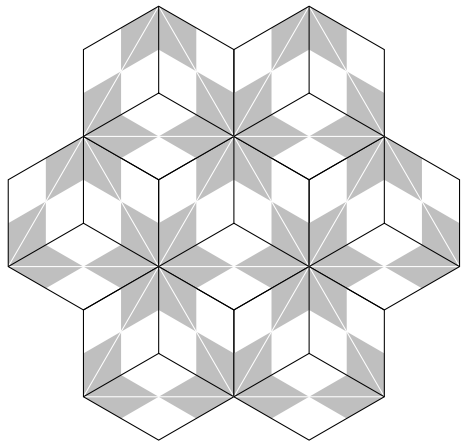
\begin{figure}[h]
\begin{tikzpicture}
\foreach \x in {0,2,4} {
  \fill[gray!50] (\x+.5,1.443) -- (\x+.5,.866) -- (\x+1,1.155) -- (\x+1,1.732) -- cycle;
  \fill[gray!50] (\x+1,1.155) -- (\x+1.5,.866) -- (\x+1.5,1.443) -- (\x+1,1.732) -- cycle;
  \fill[gray!50] (\x,0) -- (\x,.577) -- (\x+.5,0.866) -- (\x+.5,.289) -- cycle;
  \fill[gray!50] (\x,0) -- (\x+.5,-.289) -- (\x+1,0) -- (\x+.5,.289) -- cycle;
  \fill[gray!50] (\x+2,0) -- (\x+1.5,.289) -- (\x+1,0) -- (\x+1.5,-.289) -- cycle;
  \fill[gray!50] (\x+2,0) -- (\x+1.5,.289) -- (\x+1.5,.866) -- (\x+2,.577) -- cycle;
  \draw[white] (\x,0) -- (\x+1,1.732) -- (\x+2,0) -- cycle;
  \draw (\x,0) -- (\x+1,-.577) -- (\x+2,0) -- (\x+2,1.155) -- (\x+1,1.732) -- (\x,1.155) -- cycle;
  \draw (\x+1,1.732) -- (\x+1,.577) -- (\x,0)  (\x+1,.577) -- (\x+2,0);
}
\foreach \x in {1,3} {
  \foreach \y in {-1.732,1.732} {
  \fill[gray!50] (\x+.5,\y+1.443) -- (\x+.5,\y+.866) -- (\x+1,\y+1.155) -- (\x+1,\y+1.732) -- cycle;
  \fill[gray!50] (\x+1,\y+1.155) -- (\x+1.5,\y+.866) -- (\x+1.5,\y+1.443) -- (\x+1,\y+1.732) -- cycle;
  \fill[gray!50] (\x,\y) -- (\x,\y+.577) -- (\x+.5,\y+0.866) -- (\x+.5,\y+.289) -- cycle;
  \fill[gray!50] (\x,\y) -- (\x+.5,\y-.289) -- (\x+1,\y) -- (\x+.5,\y+.289) -- cycle;
  \fill[gray!50] (\x+2,\y) -- (\x+1.5,\y+.289) -- (\x+1,\y) -- (\x+1.5,\y-.289) -- cycle;
  \fill[gray!50] (\x+2,\y) -- (\x+1.5,\y+.289) -- (\x+1.5,\y+.866) -- (\x+2,\y+.577) -- cycle;
    \draw[white] (\x,\y) -- (\x+1,\y+1.732) -- (\x+2,\y) -- cycle;
    \draw (\x,\y) -- (\x+1,\y-.577) -- (\x+2,\y) -- (\x+2,\y+1.155) -- (\x+1,\y+1.732) -- (\x,\y+1.155) -- cycle;
    \draw (\x+1,\y+1.732) -- (\x+1,\y+.577) -- (\x,\y)  (\x+1,\y+.577) -- (\x+2,\y);
  }
}
\end{tikzpicture}
\caption{The rhombille tiling}\label{fig:rhombille}
\end{figure}

\subsection{Tilings of three-space}
A tiling in $\E^d$ is said to be \emph{uniform}
if it is vertex-transitive
and has uniform $d$-polytopes as tiles \cite{coxeter1940regular}.
Recall that uniform polytopes may be defined inductively,
declaring uniform polygons to be regular
and uniform polytopes of rank 3 or higher
to be vertex-transitive with uniform facets.
% The uniform 3-polytopes are the Archimedean solids. %and Platonic solids, prisms, and antiprisms

A 3-intransitive two-orbit tiling of 3-space
has regular polyhedral tiles
and is vertex-transitive,
which means that it is a uniform tiling.
\textcite{grunbaum1994uniform} 
listed all 28 uniform tilings of 3-space.
Of these, only one is two-orbit:
the tetrahedral-octahedral honeycomb, \#1 on Gr{\"u}nbaum's list, also
called ``alternated cubic,'' ``Tetroctahedrille,''
or ``octatetrahedral.''
%(in \cite[245]{johnson1966theory}, where it is denoted $0_{[4]}$.)
 This is
3-intransitive. Coxeter describes it as the unique quasiregular
honeycomb \cite[69]{coxeter1973regular}
and assigns it the modified Schl{\"a}fli 
symbol $\{3, \stak{3}{4}\}$
and an abbreviated symbol $h \delta_4$ \cite[402]{coxeter1940regular}. 
Being semiregular (with regular tiles and a vertex-transitive group),
it also appears in Gosset's list \cite{gosset1900regular}
as the ``simple tetroctahedric check.''
\textcite{monson2012semiregular} describe this tiling at length.
% (Compare \cite[\S 4.7]{coxeter1973regular}, where Coxeter describes $\sT$ as
% a quasiregular tessellation, with modified Schl{\"a}fli 
% symbol $\{3, \stak{3}{4}\}$.
% The usage of the term
% ``quasiregular'' in \cite{coxeter1973regular} implies the local alternating behaviour we focus on in this paper.)
%  --- sec. 3
It has 6 octahedra and 8 tetrahedra meeting at each vertex;
the vertex figure is a cuboctahedron.
The corresponding ``net,'' the 1-skeleton of the tiling,
is named \textbf{fcu} by crystallographers in \cite{delgado2002three},
where this tiling is conjectured to be
the unique one with transitivity 1112,
i.e.\ whose symmetry group has one orbit on vertices,
edges, and 2-faces, and two orbits on tiles.

A 2-intransitive tiling of 3-space has regular polygon 2-faces
and is vertex-transitive.
Moreover, the facets are regular or 2-intransitive two-orbit,
hence vertex-transitive.
So such a tiling is uniform;
but we already found the only two-orbit uniform tiling
and this was 3-intransitive.

A 1-intransitive tiling of 3-space
has two kinds of edge, which must alternate around a 2-face,
so each 2-face has evenly many sides.
The facets are regular or 1-intransitive two-orbit,
and the only such polyhedron
with even-sided 2-faces is the cube.
The only face-to-face tiling by cubes is the regular one.
So no such tilings exist.

A 0-intransitive tiling of 3-space
has two kinds of vertex,
and every edge must be incident to one of each
(by edge-transitivity),
so each 2-face has evenly many sides.
The facets are regular or 0-intransitive two-orbit;
the only possibilities
are the cube, the rhombic dodecahedron,
or the rhombic triacontahedron.
As we already mentioned,
the only face-to-face tiling by cubes is regular.
The rhombic triacontahedron has a dihedral angle of
$4 \pi / 5$, so it is impossible to fit
an integral number of them around an edge in 3-space.
However, the rhombic dodecahedron, with a dihedral angle
of $2 \pi / 3$, does form a two-orbit tiling of 3-space
in a unique way. %***
This tiling (called the rhombic dodecahedral honeycomb)
is dual to the tetrahedral-octahedral honeycomb above.
%It is the Voronoi diagram of the face-centered cubic packing, 
The corresponding net is named \textbf{flu} in \cite{delgado2002three},
and described as the structure of fluorite ($\text{CaF}_2$.)
It is conjectured there to be the unique tiling with transitivity 2111.

Suppose $\sT$ is a fully transitive two-orbit tiling;
then the facets are regular or two-orbit, and all of one type.
Since $\sT$ is vertex-transitive,
if the facets were regular,
$\sT$ would be uniform,
and we have already checked all the uniform tilings.
Thus the facets must be two-orbit.
Since $\sT$ is 2-face-transitive,
every 2-face is the same, which rules out the cuboctahedron or icosidodecahedron as facets.
The remaining possibilities are the rhombic dodecahedron,
which only appears in the 0-intransitive tiling already listed,
and the rhombic triacontahedron, which as mentioned does not tile 3-space.

\subsection{Higher dimensions}
For a rank $(d+1)$ tiling $\sT$ with $d \geq 4$,
the facets and vertex figures are $d$-dimensional polytopes
with at most two orbits.
Since no two-orbit convex polytopes exist in $d \geq 4$ dimensions,
the facets and vertex figures must, in fact, be regular;
but then $\sT$ itself is regular \cite[129]{coxeter1973regular}.

\section{Conclusion}
The number of half-regular convex polytopes and tilings (to use Conway's pleasant term) is perhaps surprisingly small.
Those which are also combinatorially two-orbit
are simply the cuboctahedron and the icosidodecahedron,
the only two quasiregular polyhedra, and their duals;
the trihexagonal tiling, the only quasiregular plane tiling,
and its dual;
and the tetrahedral-octahedral honeycomb,
the only quasiregular honeycomb,
and its dual.
It is notable, perhaps,
that although duality is not generally well-defined for tilings,
it always works well for two-orbit tilings which are combinatorially two-orbit,
just as it always works well for regular tilings and uniform plane tilings.
However, it does not generally work out for two-orbit tilings which 
are combinatorially regular!

The above seems suggestive that ``quasiregular,''
which has previously had rather ad-hoc definitions,
could be taken to mean ``facet-intransitive two-orbit.''
\textcite[18]{coxeter1973regular} %\S 2.3
defines
a ``quasi-regular polyhedron'' as ``having regular faces,
while its vertex figures, though not regular, are cyclic and
equi-angular (i.e., inscriptible in circles and alternate-sided).''
% It follows from this definition that the edges are all equal, say of
% length $2L$, that the dihedral angles are all equal, and that the
% faces are of two kinds, each face of one kind being entirely
% surrounded by faces of the other kind. Moreover, by a natural
% extension of the argument used for a regular polyhedron in \S 2.1,
% the vertex figures are all equal. ''
The definition of a quasiregular plane tiling does not seem
to be clearly stated,
but the implication (in \cite[\S 4.2]{coxeter1973regular}) is that
a quasiregular plane tiling is one
formed, as the quasiregular polyhedra can be,
by truncating the vertices of a regular tiling to the midpoints of the edges.
%p.60:
%By analogy with \S 2.2, we then call [$\{3,6\}$ and $\{6,3\}$]
%\emph{reciprocal} [i.e.\ dual] tessellations \dots The common midpoints
%of their edges are the vertices of the quasiregular tessellation
%$\vsch{3}{6}$, whose faces are triangles and hexagons
%arranged alternately\dots
In \cite[\S 4.7]{coxeter1973regular}, a tiling of 3-space (or honeycomb)
``is said to be quasi-regular if its cells are regular
while its vertex figures are quasi-regular.'' %\cite[69]{coxeter1973regular}
%This definition implies that the vertex figures are all alike,
%and that the cells are of two kinds, arranged alternately.
%\dots
%there is only one quasiregular honeycomb.
This suggests the beginning of an inductive definition
for ``quasiregular'' in higher dimensions,
which would perhaps agree with ours:
Facet-intransitive two-orbit polytopes
have regular facets
and the vertex figures are again facet-intransitive and two-orbit.
It would be good to establish that
having regular facets and facet-intransitive two-orbit vertex figures
implies that the polytope is two-orbit.
This is vacuously true for convex polytopes,
since \textcite{blind1979konvexe} classified all regular-faced
$d$-polytopes with $d \geq 4$,
and none have two-orbit vertex figures.
However, the corresponding result for abstract polytopes would clarify the agreement
of the definitions.
%p. 100--101, \S 6.4
%gives some quasiregular star polyhedra. Are they two-orbit, vertex-transitive?? Yes.
%$\vsch{3}{5/2}$ is isomorphic to $\vsch{3}{5}$, so (combinatorially) yes.
%$\vsch{5}{5/2}$ is isomorphic to $\vsch{5/2}{5}$, basically itself, so\dots

%\textcite[219]{johnson1966theory} defines a $d$-polytope or rank $d$ tiling $P$ to be quasiregular:
%\begin{itemize}
%\item When $d = 0$.
%\item When $d = 1$, if it is vertex-transitive.
%\item When $d \geq 2$, if there is a flag $\{F_{-1},F_0,\dotsc,F_d\}$ of $P$
%such that for each $k$, $1 \leq k \leq d - 1$,
%there is a group of symmetries of $P$ which is transitive on the
%$(k-1)$-faces of $F_{k+1}$ containing $F_{k-2}$.
%%i.e., on the edges of the polygon $F_{k+1} / F_{k-2}$.
%\end{itemize}
%% includes regular
%% implies uniform, for d \geq 3
%% quasiregular => vertex-figures are quasiregular
%% facets are regular, maybe of two kinds, alternating at each ridge
%% (so all ridges are alike)

The word ``quasiregular'' is also applied to some star polytopes,
such as the dodecadodecahedron $\vsch{5}{5/2}$
and the great icosidodecahedron $\vsch{3}{5/2}$ in \cite[100--101]{coxeter1973regular};
three ditrigonal forms: the ditrigonal dodecadodecahedron, 
small ditrigonal icosidodecahedron, and  great ditrigonal icosidodecahedron
(also called ``triambic'' instead of ``ditrigonal'');
and nine hemihedra: 
the tetrahemihexahedron,
octahemioctahedron,
cubohemioctahedron,
small icosihemidodecahedron,
small dodecahemidodecahedron,
great dodecahemicosahedron,
small dodecahemicosahedron,
great dodecahemidodecahedron, and
great icosihemidodecahedron (using names from \cite{wenninger1974polyhedron}).
All of these are two-orbit facet-intransitive. It would be good to establish that
these are the only two-orbit facet-intransitive star polytopes.

%The proposed definition of %``facet-intransitive two-orbit'' for
%``quasiregular'' %is more universal and succinct than this collection.
%%This 
%wouldn't fit any convex
%polytopes or tilings in higher dimensions, but it would be
%applicable to abstract polytopes.
%For example, the ``alternating semiregular polytopes'' 
%in \textcite{monson2012semiregular} are quasiregular in this sense.

Remaining questions include the classification of two-orbit tilings of hyperbolic space,
two-orbit star polytopes,
and other non-convex two-orbit polytopes in Euclidean space.
The general abstract two-orbit polyhedra have been addressed in \cite{hubard2010two}, with extension to higher dimensions in preparation \cite{hubardschultetwo}.
An overview is in \cite[\S 1.3]{helfand2013constructions}.
The important special case of chiral polytopes have been studied extensively
but many open questions remain; a recent survey is \cite{pellicer2012developments}.

%It is an open question whether a combinatorially two-orbit convex polytope
%must be combinatorially isomorphic to a (geometrically) two-orbit convex polytope.

It also remains to classify convex polytopes of three or more orbits.
Results in this direction,
mostly for abstract polytopes,
are found in \cite{cunningham2012orbit}, \cite{helfand2013constructions}, and
\cite{orbanic2010map}.

\section{Acknowledgments}
The author would like to thank his adviser, Egon Schulte, for his guidance and assistance,
and suggesting the original problem, and Peter McMullen for suggesting improvements.

\nocite{coxeter1985regular,johnson2000uniform}
\printbibliography
\end{document}